\documentclass{amsart}
\usepackage{url}
\usepackage{amsmath,amsfonts,euscript,amsthm,amssymb,amscd}
\usepackage[T1]{fontenc}
\usepackage[utf8]{inputenc}
\usepackage{lmodern}
\usepackage{epigraph}
\usepackage[shortlabels]{enumitem}
\usepackage{todonotes}
\theoremstyle{definition}
\newtheorem{Def}{Definition}[subsection]
\newtheorem{Def-Prop}[Def]{Definition-Proposition}

\newtheorem{Th}[Def]{Theorem}

\newtheorem{remark}[Def]{Remark}
\newtheorem{Prop}[Def]{Proposition}
\newtheorem{Lemma}[Def]{Lemma}
\newtheorem{Cor}[Def]{Corollary}

\newtheorem{Conj}{Conjecture}[subsection]

\DeclareMathAlphabet\mathbfcal{OMS}{cmsy}{b}{n}

\DeclareMathOperator*{\clim}{\text{colim}}

\newcommand{\hdot}{{\:\raisebox{3pt}{\text{\circle*{1.5}}}}}
\newcommand{\hdotc}{{\:\raisebox{1pt}{\text{\circle*{1.5}}}}}

\usepackage[all]{xy}
\usepackage{tikz}
 \tikzstyle{int}=[circle, draw,fill=black,outer sep=0,minimum size=3pt, inner sep=0]
  \tikzstyle{ext}=[circle, draw=black,outer sep=0,inner sep=1pt]
 \newcommand{\Beq}{\begin{equation}}
 \newcommand{\Eeq}{\end{equation}}
 \newcommand{\Beqr}{\begin{eqnarray}}
 \newcommand{\Eeqr}{\end{eqnarray}}
 \newcommand{\Beqrn}{\begin{eqnarray*}}
 \newcommand{\Eeqrn}{\end{eqnarray*}}
 \newcommand{\Ba}{\begin{array}}
 \newcommand{\Ea}{\end{array}}
 \newcommand{\Bi}{\begin{itemize}}
 \newcommand{\Ei}{\end{itemize}}
 \newcommand{\Bc}{\begin{center}}
 \newcommand{\Ec}{\end{center}}

\input{diag}
\tikzset{snakeit/.style={decorate, decoration={snake, amplitude=.2mm,segment length=1mm}}}

\tikzset{ext/.style={circle, draw,inner sep=1pt}, int/.style={circle,draw,fill,inner sep=2pt},nil/.style={inner sep=1pt}}
\tikzset{cy/.style={circle,draw,fill,inner sep=2pt},scy/.style={circle,draw,inner sep=2pt},scyx/.style={draw,cross out,inner sep=2pt},scyt/.style={draw,regular polygon,regular polygon sides=3,inner sep=0.95pt}}
\tikzset{exte/.style={circle, draw,inner sep=3pt},inte/.style={circle,draw,fill,inner sep=3pt}}
\tikzset{diagram/.style={matrix of math nodes, row sep=3em, column sep=2.5em, text height=1.5ex, text depth=0.25ex}}
\tikzset{diagram2/.style={matrix of math nodes, row sep=0.5em, column sep=0.5em, text height=1.5ex, text depth=0.25ex}}
\tikzset{rowcolsep/.style={column sep=.2cm, row sep=.1cm}}

\tikzset{
  crossed/.style={
    decoration={markings,mark=at position .5 with {\arrow{|}}},
    postaction={decorate},
    shorten >=0.4pt}}

\tikzset{every picture/.style={baseline=-.65ex} }

\newcommand{\dEd}{{
\begin{tikzpicture}[baseline=-.65ex,scale=.5]
 \node[nil] (a) at (0,0) {};
 \node[int] (b) at (1,0) {};
 \node[nil] (c) at (2,0) {};
 \draw (a) edge[latex-] (b);
 \draw (b) edge[-latex] (c);
\end{tikzpicture}}}

\begin{document}
\author{Alexey Kalugin} 
\address{Max Planck Institut für Mathematik in den Naturwissenschaften, Inselstraße 22, 04103 Leipzig, Germany}
\email{alexey.kalugin@mis.mpg.de}
\title{Oriented Getzler-Kapranov complexes and framed curves}

\maketitle
\begin{abstract} In the present paper we introduce and study oriented Getzler-Kapranov complexes. These complexes are generalisations of S. Merkulov's oriented graph complex. We investigate their relation to the cohomology of moduli spaces of complex and tropical curves, ribbon graph complexes and motivic structures in string topology.  

\end{abstract} 

\section{Introduction} 

\subsection{Introduction} In his groundbreaking work on quantum groups V. Drinfeld \cite{Drin} introduced a notion of a \textit{Lie bialgebra} as a classical limit of a quantum group. According to \textit{ibid.} a Lie bialgebra $\mathfrak g$ is a Lie algebra $\mathfrak g$ endowed with a structure of a Lie coalgebra, such that the cobracket:
$$
\delta\colon \mathfrak g\longrightarrow \mathfrak g\otimes \mathfrak g
$$
is a $1$-cocycle with respect to an adjoint action of $\mathfrak g$ on $\mathfrak g\otimes \mathfrak g.$ The corresponding "operadic notion" which governs a $(c,d)$-Lie bialgebra structure on a graded vector spaces  is a properad $\textsf {LieB}_{c,d}.$ Algebras over this properad have a bracket of degree $1-d$ and a cobracket of degree $1-c.$\footnote{One gets the original definition of V. Drinfeld for $d=c=1.$} According to \cite{MW2} the deformation complex of the properad $\textsf{LieB}_{d,d}$ is given by the \textit{oriented graph complex} $\textsf{OGC}_{2d+1}^{\hdot}$ i.e. chains are directed graphs without cycles and a differential which splits a vertex. The cohomology of $\textsf{OGC}_{2d+1}^{\hdot}$ has a rich structure in particular $H^{\hdot}(\textsf{OGC}_{2d+1})$ contains the famous Grothendieck-Teichmüller Lie algebra \cite{Will2}. There exists an important class of \textit{involutive} Lie bialgebras (a composition of the cobracket and bracket of a Lie bialgebra is zero), with the corresponding properad denoted by $\textsf{LieB}_{d,d}^{\diamond}.$ Involutive Lie bialgebras appear in various subjects: moduli stacks of curves, string topology, symplectic field theory \cite{MW} \cite{Tur} \cite{Sc} \cite{CFL}. 
\par\medskip 
In the seminal paper \cite{MW} S. Merkulov and T. Willwacher found a generalisation of the properad $\textsf{LieB}_{d,d}^{\diamond},$ According to \textit{ibid.} a \textit{properad of ribbon graphs} $\textsf {RGra}_d$ is a properad with $(n,m)$-operations given by vector space generated by ribbon ribbon graphs (up to a choice of the orientation) with $[m]$-labelled vertices and $[n]$-labelled boundaries. A composition rule is defined by gluing boundaries to vertices. The properad $\textsf {RGra}_d$ is equipped with a \textit{Chas-Sullivan morphism}:
$$
\diamond\colon \textsf {LieB}_{d,d}^{\diamond}\longrightarrow \textsf {RGra}_d,
$$
defined by the following rule:
$$
\diamond \colon [\,\,,\,\,]\longmapsto
\xy
 (0,0)*{\bullet}="a",
(5,0)*{\bullet}="b",
\ar @{-} "a";"b" <0pt>
\endxy \ \quad \diamond \colon \delta \longmapsto \ \xy
(0,-2)*{\bullet}="A";
(0,-2)*{\bullet}="B";
"A"; "B" **\crv{(6,6) & (-6,6)};
\endxy.
$$
\begin{figure}
\begin{tikzpicture}[baseline=-.65ex]
  \node[int] (v0) at (0:1) {};
\node[int] (v2) at (120:1) {};
\node[int] (v4) at (240:1) {};
\draw (v0) edge (v2) edge (v4) (v2) edge (v4);
\draw (v0) to[ out=180, in=90, looseness=30] (v0);
\draw (v2) to[ out=-70, in=45, looseness=30] (v2);
\draw (v4) to[ out=70, in=-45, looseness=30] (v4);
 \end{tikzpicture}
,\qquad 
\begin{tikzpicture}[baseline=-.65ex]
  \node[int] (v0) at (-45:1) {};
  \node[int] (v1) at (45:1) {};
 \draw (v0) edge (v1) edge[bend left] (v1) edge[bend right] (v1);
 \end{tikzpicture}

  \caption{\label{fig:loops} The left graph has genus one and belongs to $\textsf{RGra}_d(3,3)$ the right one has genus zero and belongs to $\textsf{RGra}_d(3,2).$}
 \end{figure}
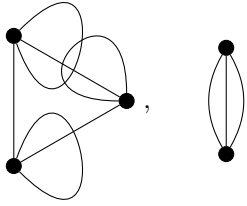
The Chas-Sullivan morphism has a very rich deformation theory \cite{MW}. Consider the composition $*\colon \textsf {LieB}_{d,d}\rightarrow \textsf {LieB}^{\diamond}_{d,d}\overset{\diamond}{\rightarrow}\textsf {RGra}_d,$ according to \textit{ibid} the corresponding deformation complex $\textsf {RGC}_d^{\hdot}(\delta+\Delta_1)$ is given by collections of ribbon graphs $\textsf {RGC}_d^{\hdot}$ (with non-labelled boundaries and vertices) equipped with a differential:
$$
\delta+\Delta_1\colon \textsf{RGC}_d^{\hdot}\longrightarrow \textsf {RGC}_d^{\hdot},
$$
where $\delta$ is a standard "splitting a vertex" differential \cite{Penn1} \cite{Kon1} and $\Delta_1$ is the so-called \textit{Bridgeland differential} \cite{Brid}. The cohomology of the latter complex is given by the totality of the shifted compactly supported cohomology of $\mathcal M_g$ \cite{AK4}.  Applying the functoriality of deformation complexes, the Chas-Sullivan morphism induces a map from the oriented graph complex $\textsf{OGC}_{2d+1}^{\hdot}$ to the ribbon graph complex $\textsf {RGC}_d^{\hdot}(\delta+\Delta_1)$. In \cite{MW} (Page $4$) the following was proposed:
\begin{Conj}[S. Merkulov and T. Willwacher '15]\label{MW} The canonical morphism:
$$H^{\hdot}(\textsf {OGC}_{2d+1})\longrightarrow H^{\hdot+1}(\textsf{RGC}_d(\delta+\Delta_1))$$
is injective.
\end{Conj} 
In the present paper we suggest a path to prove this conjecture based on the notion of an oriented Getzler-Kapranov complex. We explain a relation to the recent advances by M. Chan, S. Galatius and S. Payne \cite{CGP1} \cite{CGP2} and motivic structures in string topology studied by R. Hain \cite{Hain} \cite{Hain2}.

\subsection{Main results}  For a non empty finite set $S$ and an integer $d$ we introduce the so-called \textit{$S$-marked oriented Getzler-Kapranov complex} $\textsf H_S\textsf {OGK}_{2d+1}^{\hdot}.$ In order to do it we define the moduli spaces of oriented $S$-marked tropical curves $OM_{g,S}^{trop}$ as a \textit{directed} analog of the moduli space of $S$-marked tropical curves $M_{g,S}^{trop}$ \cite{BMV}. Recall that moduli spaces $M_{g,S}^{trop}$ are object where DG-modular (co)operads "live" naturally \cite{AK}. Analogously, moduli spaces $OM_{g,S}^{trop}$ are very close to DG-(co)properads. We consider a stable version $\overline{\mathcal N}^{fr}$ \cite{KonS} of G. Segal's properad which consists of stable nodal bordered surfaces with analytically parametrised boundaries. One can define a natural DG-combinatorial sheaf $\EuScript {ODM}_{g,S}$ on $OM_{g,S}^{trop},$ associated with $\overline{\mathcal N}^{fr}.$ These sheaves are analogous to the Deligne-Mumford sheaves from \cite{AK}. Hence we define the $S$-marked oriented Getzler-Kapranov complex by the rule:
$$
\textsf H_S \textsf {OGK}^{\hdot}_{2d+1}:=\prod_{g\geq 0\, 2g+|S|-2>0}^{\infty} \mathbf R^{\hdot-g(1+2d)-|S|}\Gamma_c(\mathcal {OM}_{g,S}^{trop},\EuScript {ODM}_{g,S})
$$
Denote by $\textsf H_S\textsf {GK}^{\hdot}_{2d}$ the $S$-marked Getzler-Kapranov complex \cite{AWZ} \cite{AK}. Our first result is:
\begin{Th}\label{ziv} For every integer $d$ and a finite non-empty set $S$ there is an explicit quasi-isomorphism of complexes (M. Živković's map):
$$
\Psi\colon \textsf H_S\textsf {OGK}_{2d+1}^{\hdot}\overset{\sim}{\longrightarrow} \textsf H_S\textsf {GK}_{2d}^{\hdot}
$$
\end{Th} 
Recall that in \cite{AK} a notion of the Getzler-Kapranov complex of weight $k$ was introduced $\textsf W_k\textsf H_S\textsf {GK}_{2d}^{\hdot}.$ These complexes are "decorated graph complexes" which compute the weight $k$-quotient of the compactly supported cohomology of moduli stacks of smooth and proper algebraic curves $\mathcal M_{g,S}$ \textit{ibid}. Analogously one can define a notion of the weight $k$ oriented Getzler-Kapranov complex $\textsf W_k\textsf H_S\textsf {OGK}_{d}^{\hdot}.$ The important property of the Živković morphism that it preserves weight quotients i.e. we have the following:
\begin{Cor}\label{ziv1} The Živković morphism preserves the weight $k$-complexes:
$$
\Psi\colon H^{\hdot}(\textsf W_k\textsf H_S\textsf {OGK}_{2d+1})\overset{\sim}{\longrightarrow} H^{\hdot}(\textsf W_k\textsf H_S\textsf {GK}_{2d})\cong \prod_{g\geq 0\, 2g+|S|-2>0}^{\infty}\mathrm {gr}^W_kH_c^{\hdot+4dg}(\mathcal M_{g,S},\mathbb Q) 
$$
\end{Cor} 
Denote by $\textsf H_S\textsf {OGC}_{2d+1}$ the $S$-marked oriented graph complex \cite{AWZ}. This complex naturally maps to $\textsf W_0\textsf H_S\textsf {OGK}_{2d+1}^{\hdot}.$ We recover one of the main results of \textit{ibid.} (Theorem $2$):
\begin{Cor}\label{ziv2} The following diagram commutes: 
\begin{equation}
\begin{diagram}[height=2.3em,width=2.3em]
 \textsf W_0\textsf H_S\textsf {OGK}_{2d+1}^{\hdot} & &  \rTo^{\Psi}_{\sim} &  & \textsf W_0 \textsf H_S\textsf {GK}_{2d}^{\hdot}   &  \\
\uTo_{\sim}^{} & & &  & \uTo_ {\sim}^{}  && \\
\textsf H_S\textsf {OGC}_{2d+1}^{\hdot} & &  \rTo^{\sim}_{} &  &  \textsf H_S\textsf {GC}_{2d}^{\hdot}\\
\end{diagram}
\end{equation}
\end{Cor}
Where  $\textsf H_S\textsf {GC}_{2d}^{\hdot}$ is a $S$-marked graph complex \cite{CGP1}. Applying known results about the cohomology of the Getzler-Kapranov complexes from Theorem \ref{ziv} one computes the cohomology of the corresponding oriented graph complexes \cite{BFP} \cite{PW}. 
\par\medskip 
Our next result relates the $S$-marked oriented Getzler-Kapranov complex to the certain deformation complex. Denote by $\overline{\textsf N}^{fr}$ the DG-coproperad defined as rational cochains of the properad $\overline{\mathcal N}^{fr}$ of stable bordered surfaces with analytically parametrised boundaries and by $\textsf {Frob}$ the properad of Frobenius algebras. We have a natural "TQFT-morphism":
$$
\star\colon \textsf {Frob}^*\longrightarrow \overline{\textsf N}^{fr}
$$
Denote by $\textsf {AC}$ the properad from \cite{AWZ} this properad is naturally equipped with a morphisma $\textsf {Frob}^{\diamond} \rightarrow \textsf {AC}\rightarrow \textsf {Frob}^{\diamond},$ where $\textsf {Frob}^{\diamond}$ is a properad of involutive Frobenius algebras. Denote by $\textsf B_g\textsf {HH}_S\textsf {OGK}^{\hdot}$ the part of the deformation complex $\mathrm {Def}(\Omega(\textsf {AC}^*) \longrightarrow  \Omega(\overline{\textsf N}^{fr})),$ which consists $S$-labelled decorated graphs of genus $g.$ Following \textit{ibid.} by $\textsf 
B_g\textsf {HH}_S\textsf {OGC}^{\hdot}$ the part of the deformation complex $\mathrm {Def}(\Omega(\textsf {AC}^*) \longrightarrow  \textsf {LieB}_0),$ which consists $S$-labelled graphs of genus $g.$ The cohomology of this complex can be identified with the cohomology of oriented graph complex $\textsf B_g\textsf {H}_S\textsf {OGC}^{\hdot}$ (Proposition $9$ \cite{AWZ}).
\begin{Th}\label{def1} For every $g\geq 0$ and non empty finite set $S$ such that $2g+|S|-2>0$ there is a canonical morphism:
$$
\textsf B_g\textsf H_S\textsf {OGC}_{d}^{\hdot}{\longrightarrow}\textsf B_g\textsf H\textsf H_S\textsf {OGK}_{d}^{\hdot},
 $$
which induces the injection in the cohomology, such that the following diagram commutes:
\begin{equation}
\begin{diagram}[height=2.3em,width=2.3em]
 H^{\hdot}(\textsf B_g\textsf H_S\textsf {OGC}) & &  \rTo^{\sim} &  &  H^{\hdot}(\textsf B_g\textsf H\textsf H_S\textsf {OGC}) &  \\
\dTo_{}^{} & & &  & \dTo_ {}^{}  && \\
 H^{\hdot}(\textsf B_g\textsf H_S\textsf {OGK}_{0}^{\hdot}) & &  \rInto^{}_{} &  & H^{\hdot}(\textsf B_g\textsf H\textsf H_S\textsf {OGK})
 \\
\end{diagram}
\end{equation}
\end{Th} 
Applying all our previous results we obtain the following:
\begin{Cor}\label{imp} For every $g\geq 0$ and non empty finite set $S$ such that $2g+|S|-2>0$ we have the injective morphism:
$$
H^{\hdot}(\textsf B_g\textsf H\textsf H_S\textsf {OGC})\longrightarrow H^{\hdot}(\textsf B_g\textsf H\textsf H_S\textsf {OGK})
$$
\end{Cor}

\subsection{Merkulov-Willwacher's conjecture} Denote by  $\textsf {HoLieb}_0^{\diamond}$ the minimal resolution of $\textsf {LieB}_0^{\diamond},$ recall that $\Omega(\textsf {Frob})\cong \textsf {HoLieb}_0^{\diamond}$ \cite{CMW}, where $\Omega$ is a properadic cobar construction. By $*\colon \textsf {HoLieB}_0^{\diamond} \rightarrow \textsf {RGra}_0$ we denote the composite morphism from the properad $\textsf {LieB}_0^{\diamond}$ which send a cobracket to zero. We give the following:
\begin{Conj}\label{MW1}
There is a morphism (rather a "roof" of morphisms) of properads:
\begin{equation}\label{tq}
\Omega(\overline{\textsf N}^{fr})\longrightarrow \textsf {RGra}_0,
\end{equation}
such that the following diagram commutes: 
\begin{equation}
\begin{diagram}[height=2.3em,width=2.3em]
\Omega(\textsf {Frob}^*) & &  \rTo^{\star} &  &  \Omega(\overline{\textsf N}^{fr})   &  \\
\dTo_{\sim}^{} & & &  & \dTo_ {\mathrm{}}^{}  && \\
\textsf {HoLieB}_0^{\diamond} & &  \rTo^{*}_{} &  & \textsf {RGra}_0   \\
\end{diagram}
\end{equation}
\end{Conj} 
Here we expect that morphism \eqref{tq} is induced by K. Costello's homotopy equivalence \cite{Cost} (see Conjecture \ref{Conj1} for a precise statement). Our next result is:

\begin{Th}\label{MW4} Conjecture \ref{MW1}, more precisely its refinement (Remark \ref{mgr}) implies the Merkulov-Willwacher conjecture (Conjecture \ref{MW}).

\end{Th} 
We apply Corollary \ref{imp} and obtain a proof of Conjecture $25$ from \cite{AWZ} (we show that the properadic morphism from the $S$-marked hairy graph complex to the ribbon graph complex (see \textit{ibid}.) coincides with the Chan-Galatius-Payne morphism \cite{CGP2}). Further applying the main results of \cite{AK4} and \cite{CGP1} we get a proof Conjecture \ref{MW}. 

\subsection{String topology} Here we explain the relation of our results to string topology:
\par\medskip 
In \cite{Vain} the logarithmic stack $\mathbfcal M_{g,k}$ which classifies formal curves with framings was introduced. These moduli stacks are equipped with gluing morphisms \textit{ibid.} and closely related to Kimura-Stasheff-Voronov spaces \cite{KSV}. The collection $\{\mathbfcal M_{g,k+p}\}$ can be made into a properad in the category of logarithmic stacks. The Kato-Nakayama analytification \cite{KaN} of the corresponding reduced log stack $\{\mathfrak m_{g,k+p}^{an}\}$ is equivalent to G. Segal's properad of $2$-bordisms $\{\mathcal N_{g,k,p}^{fr}\}.$ According to \cite{Cost1} $H_{\hdotc}(\mathcal N_{g,k,m}^{fr})$ naturally acts on the homology of a free loop space $LM$ of a simply connected and compact complex manifold $M,$ generalising the Chas-Sullivan string product in genus zero. We state the following:
\begin{Conj} 
\begin{enumerate}[(i)]
\par\medskip 
\item There exists a DG-properad $\textsf N^{fr}_{\mathrm {LMM}}$ in the category of log mixed motives.\footnote{See a discussion in \cite{Vain}.} 
\par\medskip 
\item The corresponding Hodge realisation $\textsf N^{fr}_{\mathrm {MHS}}$ in the category of mixed Hodge structures acts on the homology of $LM$ in the way compatible with the mixed Hodge structure on $H_{\hdotc}(LM)$ \cite{Hain3}. 
\par\medskip 
\item The corresponding Betti realisation $\textsf N^{fr}_{\mathrm B}:=\{\mathfrak m_{g,k+p}^{an}\}$ coincides with K. Costello's action. 
\end{enumerate} 
\end{Conj}
It would be very interesting to study the relation to the Goldman-Turaev formality \cite{AKKN} and to the "derived" motivic Galois group \cite{Ay} action on formalities. I hope to elaborate on this elsewhere.



\subsection{Structure of the paper} In Section $2$ we recollect some facts about the sheaves on diagrams (following P. Deligne \cite{SD}), moduli of tropical curves and Getzler-Kapranov complexes \cite{AK}. We also introduce a moduli space of oriented tropical curves. In Section $3$ we define and study oriented Getzler-Kapranov complexes. We give a definition of the morphism $\Psi$ and prove Theorem \ref{ziv} and Corollaries \ref{ziv1} and \ref{ziv2}. Section $4$ is devoted to give an "operadic" definition of the oriented Getzler-Kapranov complex. In this Section we prove Theorem \ref{def1} and Corollary \ref{imp} as well as explain Conjecture \ref{MW1} and speculate a little about it (Theorem \ref{MW4}). 

\subsection{Acknowledgments} This work was supported by the Max Planck Institute for Mathematics in Sciences.

\section{Preliminaries}

\subsection{Notation} For a natural number $n\in \mathbb N_+$ we will denote by $[n]$ a finite set such that $[n]:=\{1,2,\dots,n\}.$ Let $I$ be a finite set, by $\mathrm {Aut}(I)$ we will denote the group of automorphisms of this set, in the case when $I=[n]$ we will use a notation $\Sigma_n:=\mathrm {Aut}([n])$ for a symmetric group on $n$-letters. We work over the field of rational number $\mathbb Q$ (all sheaves and vector spaces are defined over $\mathbb Q$). For a $\mathbb Q$-linear representation $V$ of a finite group $G$ we will denote by $V^G$ (resp. $V_G$) the space of $G$-invariants (resp. $G$-coinvariants). Since the characteristic of a field $\mathbb Q$ is zero, the canonical morphism
$V_G\longrightarrow V^G$ is an isomorphism, and hence we will freely switch between invariants and coinvariants. For a finite $S$ we will denote by $\det(S):=\bigwedge^{\dim V\langle S\rangle} V\langle S\rangle$ the determinant on the free $\mathbb Q$-vector space $V\langle S\rangle$ generated by a set $S.$ 
\par\medskip 
For any $g\geq 0,$ and a finite set $S$ we have a moduli stack $\mathcal M_{g,S}$ which parametrises smooth proper algebraic curves of genus $g$ with $S$-labelled marked points. This is a smooth Artin stack of dimension $3g-3+|S|$ over $\mathbb Z.$ When $2g+|S|-2>0$ this is a smooth Deligne-Mumford stack. For $2g+|S|-2>0$ according to \cite{DM} there exists a compactification $\overline {\mathcal M}_{g,S}$ of $\mathcal M_{g,S}$ which is defined by adding all stable nodal curves. $\overline {\mathcal M}_{g,S}$ is a smooth and proper Deligne-Mumford stack over $\mathbb Z.$ The complement to a smooth locus is a divisor with normal crossings. 

\subsection{Sheaves on diagrams} We recall some basic facts about sheaves on diagrams following \cite{SD} and \cite{AK}. Let $A$ be category and $X_A$ be an $A$-diagram with a stratification $\mathcal S.$ We will denote by $\textsf D_{lax}(X_A,\mathcal S)$ the derived category of DG-constructible lax-sheaves on $X_A.$ We will use a notation $\textsf D(X_A,\mathcal S)$ for the triangulated category of DG-constructible sheaves on $X_A.$ With a morphism $f\colon X_A\longrightarrow Y_A$ between $A$-diagrams where each square is fibered we associate the pair of functors:
$$\mathbf R^{\hdot} g_!\colon \textsf D(X_A,\mathcal S)\longleftrightarrow \textsf D(Y_A,\mathcal S)\colon g^*.$$
Let $f\colon B\longrightarrow A$ be a functor, by $X_A\times_A B$ be denote the pullback of the $A$-digram $X_A.$ Following \cite{SD} we have a pair of adjoint functors:
$$
\mathrm {LKan}(f)\colon \textsf D_{lax}(X_A\times_A B,\mathcal S)\longleftrightarrow \textsf D_{lax}(X_A,\mathcal S)\colon \mathrm {Res}(f).
$$
Where a functor $\mathrm {LKan}(f)$ (a left Kan extension) is a left adjoint to $\mathrm {Res}(f).$ Note that generally a left Kan extension is only right exact and one has to take the left derived one, however in all cases which we will be interested in it is also left exact. We also have a notion of the derived sections with compact support: 
$$
\mathbf R^{\hdot}\Gamma_c(X_A,\,\,)\colon  \textsf D(X_A,\mathcal S)\longrightarrow \textsf D(\textsf {Vect}). 
$$

\subsection{Moduli of tropical curves and DG-sheaves} Let $g\geq 0$ and $S$ be a finite set such that $2g+|S|-2>0.$ Denote by $J_{g,S}$ a category of stable weighted $S$-marked graphs of genus $g$ \cite{CGP1} \cite{CGP2} We have a natural $J_{g,S}^{\circ}$-diagram $\mathcal M_{g,S}^{trop}$ associated with this category. The corresponding colimit $M_{g,S}^{trop}$ is called the moduli space of $S$-marked tropical curves of genus $g$ \cite{BMV}. Recall that in \cite{AK} we studied the Deligne-Mumford sheaves $\EuScript {DM}_{g,S}$ on the diagram $\mathcal M_{g,S}^{trop}.$ These DG-sheaves play a role of the geometric avatar of E. Getzler and M. Kapranov's notion of a modular operad associated with moduli stacks of stable curves \cite{KG}. For every even integer $d$ the $S$-marked Getzler-Kapranov complex $\textsf H_S\textsf {GK}^{\hdot}_d$ is defined as the derived global sections of the diagram $\mathcal M_{g,S}^{trop}$ with coefficients in $\EuScript {DM}_{g,S}.$
The Getzler-Kapranov complex  $\textsf H_S\textsf {GK}^{\hdot}_d$ computes the compactly supported cohomology of the moduli stack of smooth curves $\mathcal M_{g,S}$ \cite{AK} and serves as the geometric counterpart of the Feynmann transform defined by E. Getzler and M. Kapranov \cite{KG}. For each $k\geq 0$ the DG-combinatorial sheaves $\mathrm {Gr}^{\EuScript W}_k\EuScript{DM}_{g,S}$ on $\mathcal M_{g,S}^{trop}$ were defined in \cite{AK}. These combinatorial sheaves are designed in the way that the corresponding DG-vector space of the derived global sections with compact support $\textsf W_k\textsf H_S\textsf {GK}^{\hdot}_d$ computes the weight $k$ quotient of $H_c^{\hdot}(\mathcal M_{g,S},\mathbb Q).$ The case of the zero weight will be essentially important to us. Namely in this case we have the quasi-isomorphism between $\textsf W_0\textsf H_S\textsf {GK}^{\hdot}_d$ and the marked graph complex $\textsf H_S\textsf {GC}_d^{\hdot}$ \cite{CGP2}.









\subsection{Oriented tropical curves} By a \textit{directed graph} $G$ we understand a connected directed graph with loops and parallel edges allowed. A set of vertices will be denoted by $V(G).$ A set of edges will be denoted by $E(G).$ For each vertex $v$ we will denote by $h(v)$ a set of \textit{half-edges} attached to $v.$ We have a natural decomposition $h(v)=h_{in}(v)\sqcup h_{out}(v),$ where $h_{in}(v)$ is a set of \textit{incoming half-edges} and $h_{out}(v)$ is a set of \textit{out-going half-edges}. We say that a directed graph $G$ is \textit{oriented} if $G$ does not contain directed cycles of edges in particular it means that we do not have "wheels". By a \textit{weighted $n$-marked oriented graph} we understand a triple $(G,w,m_{out})$ where $G$ is an oriented graph together with a function $w\colon V(G)\longrightarrow \mathbb N$ called the \textit{weight (genus) function} and a function $m_{out}\colon S\longrightarrow V(G)$ called  $S$-\textit{marking}. For a vertex $v$ in a weighted marked graph we will use the following notation $n_{in}(v):=h_{in}(v)$ (resp. $n_{out}(v):=h_{out}(v)\sqcup m^{-1}_{out}(v)$ ). The \textit{genus} of a weighted $n$-marked directed graph $(G,w,m_{in},m_{out})$ is defined by the standard formula:
 $$g:=b_1(G)+\sum_{v\in V(G)} w(v),$$ 
 where $b_1(G)=|E(G)|-|V(G)|+1$ is the first Betti number of $G,$ considered as a one dimensional CW complex. We demand that there are no passing vertices of the zero weight, $m_{out}(v)$ is non empty for every $v\in V(G)$ and each vertex of the zero weight is at least two valent. It is convenient to represent markings as a directed marked half-edges (hairs) attached to a vertex with a flow directed outwards a vertex.
Oriented weighted $n$-marked graphs of genus $g$ naturally form a category denoted by $OJ_{g,S}.$ Objects of this category are oriented weighted $S$-marked graphs and morphisms are given by compositions of contractions of edges and isomorphisms which preserve markings and genus labelings (cf. \cite{KG} \cite{CGP1}). Let us explain what do we mean by the contraction of edges. There are type of such contractions. First is that we have a directed edge $e$ joining two vertices $v$ and $v'$ such that there are no parallel edges to $e$. Then we contract this edge and the resulting vertex $v'$ is assigned with a weight $w+w'.$ If $e$ is an edge joining two vertices with additional $l$ parallel edges attached we remove this edges and equip the resulting vertex with a weight $w+w'+l$ (here all edges are have the same flow by the oriented condition). 
$$
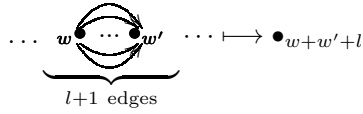
\dots\underbrace{
\Ba{c}
{\xy
(5,0)*{...},
   \ar@/^1pc/(0,0)*{{}_w \bullet};(10,0)*{\bullet_{w'}}
   \ar@/^{-1pc}/(0,0)*{{}_{w}\bullet};(10,0)*{\bullet_{w'}}
   \ar@/^0.6pc/(0,0)*{{}_{w}\bullet};(10,0)*{\bullet_{w'}}
   \ar@/^{-0.6pc}/(0,0)*{{}_{w}\bullet};(10,0)*{\bullet_{w'}}
 \endxy}
 \Ea}_{l+1\ \mathrm{edges}}\dots \longmapsto \bullet_{w+w'+l}
 $$
A collection of two parallel edges with a two valent source vertex of weight zero attached will be called an \textit{oriented loop}:
$$
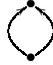
\Ba{c}\resizebox{6mm}{!}  {\xy
(0,5)*{},
   \ar@/^1pc/(0,0)*{\bullet};(0,10)*{\bullet}
   \ar@/^{-1pc}/(0,0)*{\bullet};(0,10)*{\bullet}
 \endxy}
 \Ea
 $$
For an oriented $S$-marked graph $G$ we will denote by $\mathrm{Aut}(G):=\mathrm {Isom}_{OJ_{g,S}}(G,G)$ a group of automorphisms of $G.$ For an oriented $S$-marked graph $G$ and $e\in E(G)$ we will denote by $G_e$ a graph that is obtained from $G$ by contracting an edge $e.$ The category $OJ_{g,n}$ has the terminal object denoted by $\bullet_{g,S}:$ 
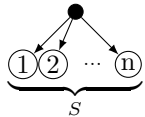
$$
\underbrace{
\begin{tikzpicture}[baseline=-.65ex]
\node[int] (v) at (0,.5) {{}};
\node[ext] (h1) at (-.7,-.2) {1};
\node[ext] (h2) at (-.3,-.2) {2};
\node[ext] (h3) at +(.7,-.2) {n};
\draw (v) edge[-latex] (h1) edge[-latex] (h2)  edge[-latex] (h3) +(.25,-.7) node {$\scriptstyle \dots$};
\end{tikzpicture}
}_{S}
$$
By $OJ_{g,n}^k$ we will denote a set of oriented $S$-marked graphs $G$ with $E(G)=k.$

\par\medskip

\begin{Def}\label{trop1} Fix $g\geq 0$ and a non empty finite set $S$ with a condition $2g-2+|S|>0.$ We define a $OJ_{g,S}$-diagram $\mathcal {OM}^{trop}_{g,S}$ called a \textit{moduli diagram of oriented tropical curves of genus $g$ with $S$-markings}:
$$
\mathcal {OM}^{trop}_{g,S}\colon OJ_{g,S}^{\circ}\longrightarrow \textsf {Top}
$$
By the following rule:
\par\medskip
For every object $(G,w,m_{out})\in OJ_{g,S}$ we set:
$$
\mathcal {OM}^{trop}_{g,S}\colon (G,w,m_{out})\longmapsto \mathbb R_{\geq 0}^{V(G)}:={\sigma}_{G}^{or}.
$$
For a morphism $f\colon (G,w,m_{out})\longrightarrow (G',w',m'_{out})$ we set:
$$
\mathcal {OM}^{trop}_{g,S}\colon f\longmapsto \sigma f,
$$
where $\sigma f\colon \sigma_{G'}^{or}\longrightarrow \sigma_{G}^{or}$ is a map that sends a $v'$-coordinate of $\mathbb R_{\geq 0}^{V(G')}$ to a $v$-coordinate of $\mathbb R_{\geq 0}^{V(G)}$ if $f$ sends (bijectively) all incoming half-edges $h_{in}(v)$ attached to a vertex $v$ to to incoming half-edges $in(v')$ attached to $v'$ and zero otherwise.
\end{Def}
Note that the colimit of the $OJ_{g,S}^{\circ}$-diagram above $OM_{g,S}^{trop}:=\clim_{OJ_{g,S}^{\circ}}\mathcal {OM}_{g,S}^{trop}$ will be called a \textit{moduli space of oriented tropical curves of genus $g$ with $S$-markings} (analogous to \cite{CGP2} \cite{CGP1}). Each space $\sigma_{G}^{or}$ carries a natural stratification by strata $\overline{S}_{G'}^{or}:=\mathbb R^{V(G')}_{\geq 0},$ where $G'$ is a graph such that there is a morphism $f\colon G\rightarrow G'$ in a category $OJ_{g,S}$  (such stratum is of dimension $|V(G')|.$) Moreover:
$$
S_{G'}^{or}\subset \overline{S}_{G''}^{or} \Leftrightarrow G'' \rightarrow G'.
$$
We will denote this stratification by $\mathcal S_{G}^{or}$ and the resulting stratification on $\mathcal {OM}_{g,S}^{trop}$ will be denoted by $\mathcal S_{or}.$
\begin{remark} Recall that a moduli space of tropical curves $M_{g,S}^{trop}$ is closely related to the Berkovich analytification of the coarse Deligne-Mumford moduli space $\overline{M}_{g,S}$ \cite{ACP}. Our motivation to introduce moduli spaces $\mathcal {OM}_{g,S}^{trop}$ has a completely a "properadic origins" and inspired by the deformation complex of the properad of involutive Lie bialgebras (see proof of Proposition \ref{alt1}). We \textit{do not} know other interpretations of $OM_{g,S}^{trop}$ and It would be interesting to find one.

\end{remark}

\section{Oriented Getzler-Kapranov complexes}

\subsection{Framed surfaces} 
Let $C$ be a bordered surface with at most nodal singularities in the interior of the surface i.e. each component of a normalisation is compact Riemann surface with boundary locally analytically modelled on the upper half plane $\{z\in \mathbb C\colon \mathrm{Im}\,z \geq  0\}.$  We say that a nodal bordered surface $C$ is a \textit{framed nodal surface} if $C$ has an analytic parametrisation $\phi_i\colon S^1 \overset{\sim}{\rightarrow} \partial C_i$ for each boundary component $\partial C_i.$ A component $\partial C_i$ is called an \textit{input} or an \textit{output} if the orientation induced by the parametrisation coincides, respectively is opposite to the boundary orientation of $\partial C_i.$ By $\mathbb D$ we will denote the standard unit disks with a framing defined by the rule $\theta\mapsto \exp(2\pi i \theta),$ by $\overline{\mathbb D}$ we will denote denote the unit disk with reversed framing i.e. $\theta\mapsto \exp(-2\pi i \theta).$ We say that a framed surface $C$ is \textit{stable} if each (non-bordered) component of the normalisation is stable in the sense of P. Deligne and D. Mumford \cite{DM}.
\begin{Def} By $\overline{\mathcal N}_{g,I,J}^{fr}$ we denote a \textit{moduli space of stable framed surfaces of genus $g$ with $I$-input framings and $J$-output framings}.

\end{Def} 
For an oriented graph $G\in OJ_{g,S}$ we denote by $\overline{\mathcal N}^{fr}(G)$ the following moduli space:
$$
\overline{\mathcal N}^{fr}(G):=\prod_{v\in V(G)} \overline{\mathcal N}_{w(v),n_{in}(v),n_{out}(v)}^{fr}
$$
Analogous to the case of Deligne-Mumford moduli stacks we have \textit{gluing morphisms} (clutching morphisms). Let $G\rightarrow G'$ be a morphism of oriented graphs i.e. a morphism in $OJ_{g,S}.$ We have a morphism:
\begin{equation}\label{glue}
glue_{G\,G'}\colon \overline{\mathcal N}^{fr}(G)\longrightarrow \overline{\mathcal N}^{fr}(G').
\end{equation} 
This morphism is defined by gluing of framed surfaces along boundary (see page $17$ in \cite{AV} ) and stabilising the resulting component. 

Recall that with a stable $I$-marked weighted graph $G\in J_{g,I}$ one can associate the proper and smooth Deligne-Mumford stack $\overline{\mathcal M}(G):=\prod_{v\in V(G)} \overline{\mathcal M}_{w(v),n(v)}.$ This stack is equivalent to the normalisation of the closure of the stratum in $\overline{\mathcal M}_{g,n},$ which consists of stable curves with a dual graph being $G.$ For a stable $I$-marked weighted graph $G\in J_{g,S}$ by a \textit {spanning forest} $\tau$ we understand the subgraph (it can be non-connected) which contains all vertices of $G$ (markings), such that each connected component of $\tau$ contains exactly one marking and no cycles.  With a pair $(G,\tau)$ we can associate a directed graph $G_{\tau}$ by the following rule: 
\begin{itemize}
\item We put a flow on each connected component of a spanning forest which goest towards the marking. 
\item All other edges are replaced by the graphs:
$$\dEd$$
\end{itemize}

With an element $C\in \overline{\mathcal M}_{G}$ and a spanning tree $\tau$ we can associate a framed surface $Fr_{\tau}(C)$ by the following rule. For a component $C(v)$ at a vertex $v$ we glue at each marked point $h$ a nodal disks such that if $h\in n_{in}(v)$ for $G_{\tau}$ we glue $\mathbb D$  and if $h\in n_{out}(v)$ we glue $\overline{\mathbb D}$ if $h$ does not belong to the spanning tree $\tau$ we attach a nodal disk $\mathbb D.$ Further at each vertex $v\in G_{\tau}$ which does not belong to a spanning tree we we take analytically parametrised nodal annulus (genus zero nodal bordered surface with two outgoing boundaries). It is easy to see that this construction extends to the the morphism between $\overline{\mathcal M}(G)$ and  $\overline{\mathcal N}^{fr}(G)$ which we shall denote by $Fr_{\tau}$ (cf. \cite{AV}).
\begin{Lemma}\label{KL1} For a spanning tree $\tau$ of $G$ we have a homotopy equivalence of moduli spaces:
\begin{equation}\label{fr}
Fr_{\tau}\colon  \overline{\mathcal M}({G})\overset{\sim}{\longrightarrow}\overline{\mathcal N}({G_{\tau}})^{fr}.
\end{equation} 
with a property:
\begin{equation}
\begin{diagram}[height=2.7em,width=2.7em]
\overline{\mathcal N}({G_{\tau}})^{fr} & &  \rTo^{glue_{G_{\tau}\,G'_{\tau}}} &  &  \overline{\mathcal N}({G_{\tau}'})^{fr} &  \\
\uTo^{Fr_{\tau}}_{\sim} & & &  & \uTo_ {Fr_{\tau}}^{\sim }  && \\
 \overline{\mathcal M}({G})  & &  \rTo^{\xi_{G\, G'}} &  &  \overline{\mathcal M}({G'})  \\
\end{diagram}
\end{equation}

\end{Lemma} 
\begin{proof} Following Lemma $6.4$ from \cite{AV} we have a morphism:
$$
cap\colon \overline{\mathcal N}({G_{\tau}})^{fr}\longrightarrow  \overline{\mathcal M}({G})
$$
which is defined by removing nodal disks and gluing $\mathbb D$ at each incoming half-edge and $\overline{\mathbb D}$ at each outgoing half-edge. It is easy to see that $cap\circ Fr=\mathrm {Id}.$ Further following \textit{ibid.} one can prove that this morphism is a deformation retract. 

\end{proof} 

\subsection{Combinatorial DG-sheaves on $\mathcal {OM}_{g,S}^{trop}$}

For a stratified space $\sigma_{G}^{or}$ associated with a weighted oriented $S$-marked graph $G\in OJ_{g,S}$ we define an $\mathcal S_{G}^{or}$-smooth DG-combinatorial sheaf  $\EuScript {ODM}_{G}\in \textsf {Sh}_c(\sigma_{G}^{or},\mathcal S_{or})$ using the standard gluing construction for combinatorial DG-sheaves on the stratified spaces with contractible strata (Proposition $1.8$ in \cite{KaSh}). Namely we set:
\begin{itemize}
  \item \par\medskip
For every graph $G'\in OJ_{g,S,G}^{\circ}$ we set:
$$
R\Gamma(S_{G}^{or},\EuScript {ODM}_{G}):=\bigotimes_{v\in V(G')} C^{\hdot}(\overline{\mathcal N}_{w(v),inv(v),out(v)}^{fr},\mathbb Q)
$$
\par\medskip
  \item For every inclusion of strata $S_{G'}^{or} \subset \overline{S}_{G''}^{or}$ we define a variation map:
$$
var_{G', G''}\colon \mathbf R^{\hdot}\Gamma(S_{G'}^{or},\EuScript{ODM}_{G})\longrightarrow \mathbf R^{\hdot}\Gamma(S_{G''}^{or},\EuScript{ODM}_{G}),\quad var_{G', G''}:=glue_{G',G''}^*
$$
as being induced by a pullback along the gluing morphism.
\end{itemize}
Since gluing maps satisfy the associativity condition we have the well-defined combinatorial DG-sheaf on $\sigma_{G}^{or}.$ Hence we give:
\begin{Def} For every $g\geq 0$ and a non-empty finite set $S$ with a condition $2g-2+|S|>0$ we define an $\mathcal S_{or}$-smooth combinatorial DG-sheaf $\EuScript{ODM}_{g,S}:=\{\EuScript{ODM}_{G}\}_{G\in OJ_{g,S}}$ on the $OJ_{g,S}^{\circ}$-diagram $\mathcal {OM}^{trop}_{g,S},$ called a \textit{framed Deligne-Mumford DG-sheaf} with connecting quasi-isomorphisms:
$$
\alpha(m)\colon m^*\EuScript{ODM}_{G}\overset{\sim}{\longrightarrow}\EuScript{ODM}_{G'},\quad m\colon G\longrightarrow G'
$$
\end{Def}

\par\medskip

Applying the truncation functors we get a sequence of DG-combinatorial sheaves on $\mathcal {OM}_{g,S}^{trop}$ (cf. \cite{AK}):
$$
\EuScript W_0\EuScript{ODM}_{g,S}\rightarrow \EuScript W_1\EuScript{ODM}_{g,S}\rightarrow\dots \rightarrow \EuScript{ODM}_{g,S}
$$
\begin{Def} For every $g\geq 0$ and a non-empty finite set $S$ with a condition $2g-2+|S|>0$ we define an $\mathcal S_{or}$-smooth DG-combinatorial sheaf $\mathrm {Gr}^{\EuScript W}_k\EuScript{ODM}_{g,S}$ on the $OJ_{g,S}^{\circ}$-diagram $\mathcal {OM}_{g,S}^{trop}$ called a \textit{framed Deligne-Mumford sheaf of the weight $k$} by the rule:
$$
\mathrm {Gr}^{\EuScript W}_k\EuScript{ODM}_{g,S}:=\mathrm {Cone}(\EuScript W_{k-1}\EuScript{ODM}_{g,S}\rightarrow \EuScript                             W_{k}\EuScript{ODM}_{g,S})
$$

\end{Def}




\subsection{Oriented Getzler-Kapranov complexes} 

Following \cite{AK} we give the following:

\begin{Def} For $d\in \mathbb Z$ we define the \textit{oriented $S$-marked Getzler-Kapranov complex} $\textsf H_S\textsf {OGK}_{d}^{\hdot}$ by the rule:
$$
\textsf H_S\textsf {OGK}_{d}^{\hdot}:=\prod_{g\geq 0\, 2g+|S|-2>0}^{\infty} \mathbf R^{\hdot+g(1-d)-|S|}\Gamma_c(\mathcal {OM}_{g,S}^{trop},\EuScript {ODM}_{g,S})
$$

\end{Def}

We can also define the weight $k$ complexes: 

\begin{Def} For every non empty set $S$ and an odd integer $d$ we denote by $\textsf W_k\textsf H_S\textsf{OGK}_{d}^{\hdot}$ a decorated graph complex which computes the cohomology with compact support of the $OJ_{g,n}^{\circ}$-diagram $\mathcal {OM}_{g,n}^{trop}$ with coefficients in the DG-combinatorial sheaf $\mathrm {Gr}^{\EuScript
W}_k\EuScript{ODM}_{g,n}:$                                                                                                                             
$$
\prod_{g\geq 0\, 2g+|S|-2>0}^{\infty}\mathbf R^{\hdot+g(1-d)-|S|}\Gamma_c(\mathcal {OM}^{trop}_{g,n},\mathrm {Gr}^{\EuScript W}_k\EuScript{ODM}_{g,n}):=\textsf W_k\textsf H_S\textsf{OGK}_{d}
$$
This DG-vector space will be called the \textit{oriented Getzler-Kapranov complex of the weight $k$.}
\end{Def}

Denote by $j\colon OM_{g,S,w=0}^{trop} \hookrightarrow  OM_{g,S}^{trop}$ the open locus of directed graphs where all vertices have zero weights attached. Following \cite{AWZ} (cf. \cite{Will2}) for a non empty set $S$ and an odd integer $d$ we define the \textit{$S$-marked oriented graph graph complex} $\textsf {H}_S\textsf {OGC}^{\hdot}_d$ by the rule:
\begin{equation}\label{ogc}
\textsf {H}_S\textsf {OGC}^{\hdot}_d:=\prod_{g\geq 0\, 2g+|S|-2>0}^{\infty} C_c^{\hdot+g(1-d)-|S|}(OM_{g,S,w=0}^{trop},\mathbb Q) 
\end{equation}
Applying the Cousin resolution one can realises this complex explicitly. Indeed element are given by pairs $(G,or)$ where $G$ is an oriented graph such that all vertices has weight zero and are least bivalent and have at least one outgoing edge also there are no passing vertices, i.e. bivalent vertices with one incoming and one outgoing edge. An orientation $or\in \det(V(G))$ and the differential is given by splitting a vertex. The grading is given by the rule $|G|=V(G)+g(1-d)-1.$ Hence Definition \ref{ogc} coincided with \cite{AWZ}. We have the following:
\begin{Prop}\label{alt1} For every non empty set $S$ and an odd integer $d$ we have a canonical quasi-isomorphism:
$$
j_!\colon \textsf H_S\textsf {O}\textsf {GC}^{\hdot}_d\overset{\sim}{\longrightarrow} \textsf W_0\textsf H_S\textsf {OGK}_d^{\hdot}.
$$

\end{Prop} 

\begin{proof} 
 The result above alternatively can be proved in the following way. Let $\textsf {hoLieB}^{\diamond}$ the minimal resolution of the properad of involutive Lie bialgebras, $\textsf {hoLieB}$ is the minimal properad of Lie bialgebras \cite{CMW} (see Section $3$ for details). We have a canonical sequence of morphisms:
$$
\Omega(\textsf {AC}^*)\longrightarrow \textsf {hoLieB}\longrightarrow \textsf {hoLieB}^{\diamond}
$$
We have a morphism of deformation complexes:
\begin{equation}\label{wg}
\mathrm{Def}(\Omega(\textsf {AC}^*)\longrightarrow \textsf {hoLieB})\longrightarrow \mathrm{Def}(\Omega(\textsf {AC}^*)\longrightarrow \textsf {hoLieB}^{\diamond})
\end{equation}
In \cite{AWZ} it was proved that the left-hand side complex is quasi-isomorphic to $\textsf H_S\textsf {O}\textsf {GC}^{\hdot}_1$ (strictly we have a quasi-isomorphism between the corresponding genus $g$ parts with $S$-markings. Analogously one can show that the complex on the right-hand side is quasi-isomorphic to $\textsf W_0\textsf H_S\textsf {OGK}_d^{\hdot}.$
From \cite{MW2} quasi-isomorphism follows. 
\end{proof}

\subsection{M. Živković's morphism} Here we construct a quasi-isomorphism between the oriented Getzler-Kapranov complex and $\textsf H_S\textsf {GK}_d^{\hdot}$ following the original construction from \cite{Ziv} (see also \cite{AWZ}):
\par\medskip 

Denote by $TJ_{g,S}$ a category with objects being pairs $(G,\tau),$ where $G$ is an element of $J_{g,S}$ and $\tau$ is a weighted spanning forest of $G.$ Morphisms in this category are defined in the evident way. We have the following correspondence:
\begin{equation}\label{corr}
\begin{diagram}[height=2.0em,width=2.5em]
 TJ_{g,S}  &   \rTo^{{p}}  &   J_{g,S} &  \\
\dTo^{z} & &  && \\
OJ_{g,S} &      &    \\
\end{diagram}
\end{equation}
Where a functor $p$ is defined by forgetting a weighted spanning tree. The functor $z$ is defined by the following rule: for a connected component of a weighted spanning tree $\tau$ we put a flow on edges going towards the marking. All other edges are replaced by edges with two outputs and no inputs. Note that slice (resp. coslice) categories for a morphism $p$ are finite groupoids. From correspondence \eqref{corr} we get the following correspondence of diagram:
\begin{equation}\label{corr2}
\begin{diagram}[height=2.6em,width=2.6em]
\textsf D_{lax}(\mathcal {M}_{g,S}^{trop}\times_{J_{g,S}}TJ_{g,S},\mathcal S_{or}) & &  \rTo^{{\pi_!}}  &  &  \textsf D_{lax}(\mathcal M_{g,S}^{trop},\mathcal S_{or})\\
\dTo^{\rho_!} & & & & & \\
\textsf D_{lax}(\mathcal {OM}_{g,S}^{trop},\mathcal S_{or}) & &  &    \\
\end{diagram}
\end{equation}
Where $\rho_!:=LKan(z)\circ q_*$ and a morphism $\pi_!$ is defined by the following rule $\pi_1:=LKan(p)$ where:
$$q\colon \mathcal {M}_{g,S}^{trop}\times_{J_{g,S}}TJ_{g,S}\longrightarrow \mathcal {OM}_{g,S}^{trop}\times_{OJ_{g,S}}TJ_{g,S}$$
is a closed inclusion of $TJ_{g,S}^{\circ}$-diagrams defined by the following rule:
\begin{equation}\label{inct}
q_{G,\tau}\colon \mathbb R_{\geq 0}^{E(G)}\hookrightarrow \mathbb R_{\geq 0}^{V(G_{\tau})}
\end{equation} 
where $q_G$ is defined by the following rule:
\begin{enumerate}
\item If $e$ belongs to the edges of the spanning tree $\tau$ we map it to the coordinate $v$ of $\mathbb R_{\geq 0}^{V(G_{\tau})}$ which correspond to the target vertex of $e.$ 
\item If $e$ does not belong to the edges of the spanning tree $\tau$ we map it to the coordinate $v$ of $\mathbb R_{\geq 0}^{V(G_{\tau})}$ which corresponds to the unique vertex of $\dEd.$
\item The unique source vertex of the connected component of the spanning forest goes to zero. 
\end{enumerate} 
The functors from \eqref{corr2} posses adjoint functors $\pi^*:=Res(p)$ and $\rho^*:=q^*\circ Res(z).$ 

\begin{remark} The colimit of the moduli $TJ_{g,S}^{\circ}$-diagram $\mathcal {M}_{g,S}^{trop}\times_{J_{g,S}}TJ_{g,S}:=\mathcal {TM}_{g,S}^{trop}$ will be denoted $TM_{g,S}^{trop}$ and will be called moduli spaces of $S$-marked tropical curves with spanning forests. From \eqref{corr} we have a diagram of topological spaces:
\begin{equation}\label{corr}
\begin{diagram}[height=2.0em,width=2.5em]
TM_{g,S}^{trop}  &   \rTo^{{}}  &   M_{g,S}^{trop} &  \\
\dTo^{} & &  && \\
OM_{g,S}^{trop} &      &    \\
\end{diagram}
\end{equation}

\end{remark}

We have the following:
\begin{Prop} We have the canonical quasi-isomorphism of DG-combinatorial sheaves on $TJ_{g,S}$-diagram $\mathcal {TM}_{g,S}^{trop}:$ 

\begin{equation}\label{id}
\mathrm {Fr}\colon \rho^*\EuScript {ODM}_{g,S}\overset{\sim}{\longrightarrow} \pi^*\EuScript {DM}_{g,S}
\end{equation} 

\end{Prop} 

\begin{proof} Consider the pullback of the framed Deligne-Mumford DG-sheaf $\EuScript {ODM}_{g,S}$ to the diagram $\mathcal {OM}_{g,S}^{trop}\times_{OJ_{g,S}}TJ_{g,S}.$ Then the values of it on the element $(G,\tau)$ is given by the $\EuScript {ODM}_{\sigma_{G_{\tau}}^{or}}.$ Then its pullback to the digram $\mathcal {TM}_{g,S}^{trop}$ is given by the same DG-sheaf because inclusion \eqref{inct} are transversal to the stratification. The space of derived sections of the DG-sheaf $\pi^*\EuScript {DM}_{g,S}$ is given by $C^{\hdot}(\overline{\mathcal M}_G,\mathbb Q)$ For a spanning forest $\tau$ we have a morphism \ref{fr} and we take the corresponding pullback:
$$
Fr^*\colon C^{\hdot}(\overline{\mathcal N}_{G_{\tau}}^{fr},\mathbb Q)\longrightarrow C^{\hdot}(\overline{\mathcal M}_{G},\mathbb Q)
$$
This morphism defines the quasi-isomorphism of DG-sheaves by Lemma \ref{KL1}.

\end{proof} 

\begin{Def} For every even $d\in \mathbb Z$ and a non empty finite set $S$ we have a morphism of complexes:
$$
\Psi\colon \textsf H_S\textsf {OGK}_{d+1}^{\hdot} \longrightarrow \textsf H_S\textsf {GK}_{d}^{\hdot}
$$
which is defined by the rule:
\begin{equation*}
\begin{diagram}[height=3.1em,width=3.1em]
\mathbf R^{\hdot}\Gamma_c(\mathcal {TM}_{g,S}^{trop},\rho^*\EuScript {ODM}_{g,S}) & &  \rTo_{\sim}^{\mathrm {Fr}} &  &  \mathbf R^{\hdot}\Gamma_c(\mathcal {TM}_{g,S}^{trop},\pi^*\EuScript {DM}_{g,S}) &  \\
\uTo_{\mathrm {adj}}^{} & & &  & \dTo_ {\mathrm {adj}}^{}  && \\
\mathbf R^{\hdot}\Gamma_c(\mathcal {OM}_{g,S}^{trop},\EuScript {ODM}_{g,S})  & &  \rDotsto{\Psi} &  & \mathbf R^{\hdot}\Gamma_c(\mathcal {M}_{g,S}^{trop},\EuScript {DM}_{g,S})  \\
\end{diagram}
\end{equation*}

\end{Def} 
By the definition we have the following:
\begin{Lemma} For every  even integer $d$ non empty finite set $S$ and $k\geq 0$ the morphism $\Psi$ preserves the Getzler-Kapranov complexes of the weight $k:$
$$
\Psi\colon \textsf W_{k}\textsf H_S\textsf {OGK}_{d+1}^{\hdot} \longrightarrow \textsf W_k\textsf H_S\textsf {GK}_{d}^{\hdot}.
$$

\end{Lemma} 

\begin{proof} 

The claims is obvious.

\end{proof} 

Recall that there is a canonical morphism $\textsf H_S\textsf {GC}_d^{\hdot}\longrightarrow \textsf W_0\textsf H_S\textsf {GC}_d^{\hdot}$ (See \cite{AK}). This morphism is a quasi-isomorphism. Hence we have the following:
\begin{Cor} The following square commutes:
\begin{equation}
\begin{diagram}[height=2.3em,width=2.3em]
 \textsf W_0\textsf H_S\textsf {OGK}_{2d+1}^{\hdot} & &  \rTo^{\Psi}_{} &  & \textsf W_0 \textsf H_S\textsf {GK}_{2d}^{\hdot}   &  \\
\uTo_{\sim}^{} & & &  & \uTo_ {\sim}^{}  && \\
\textsf H_S\textsf {OGC}_{2d+1}^{\hdot} & &  \rTo^{\sim}_{} &  &  \textsf H_S\textsf {GC}_{2d}^{\hdot}\\
\end{diagram}
\end{equation}
At the bottom we have a morphism from \cite{AWZ}.
\end{Cor} 

\begin{proof} 

The claims trivially follows from the definition of the Živković morphism.

\end{proof} 

\begin{Th} For every even $d$ the Živković morphism:
$$\Psi\colon \textsf H_S\textsf {OGK}_{d+1}^{\hdot} \overset{\sim}\longrightarrow \textsf H_S\textsf {GK}_{d}^{\hdot}$$
 is a quasi-isomorphism. 
\end{Th} 
\begin{proof} It is enough to prove the statement for $d=0.$ First one has to show that there is have a quasi-isomorphism between $\textsf {OGK}_{d+1}^{\hdot}$ and its skeletal version. The key ingredient here is Lemma $3.6$ from \cite{AV}. We consider the filtration by the number of vertices and pass to the associated graded complex.  Then analogously to \cite{AWZ} one has a result. The details shall appear elsewhere. 

\end{proof} 
Recall that for each $i$ the compactly supported cohomology $H_c^{i}(\mathcal M_{g,S},\mathbb Q)$ carries P. Deligne's mixed Hodge structure \cite{DelH} \cite{DelH2}. In particular there is a natural weight filtration $W_0H_c^{i}(\mathcal M_{g,S},\mathbb Q)\subset \dots \subset W_{i-1}H_c^{i}(\mathcal M_{g,S},\mathbb Q)\subset H_c^{i}(\mathcal M_{g,S},\mathbb Q)$ with the corresponding $k$-quotient denoted by $\mathrm {gr}_k^WH_c^{i}(\mathcal M_{g,S},\mathbb Q).$ Applying Theorem $3.4.2$ from \cite{AK} we get the following:
\begin{Cor}\label{we} For every even $d\in \mathbb Z$ a non empty finite set $S$ we have the following isomorphisms:
$$
H^{\hdot}(\textsf H_S\textsf {OGK}_{d+1})\cong \prod_{g\geq 0\,2g+|S|-2>0}^{\infty} H_c^{\hdot+2dg-|S|}(\mathcal M_{g,S},\mathbb Q) 
$$
$$
H^{\hdot}(\textsf W_k \textsf H_S\textsf {OGK}_{d+1})\cong \prod_{g\geq 0\,2g+|S|-2>0}^{\infty} \mathrm {gr}_k^WH_c^{\hdot+2dg-|S|}(\mathcal M_{g,S},\mathbb Q) 
$$
\end{Cor} 

\begin{proof} 

The claims is obvious.

\end{proof} 
\begin{remark}[Non-marked version of $\textsf H_S\textsf {OGK}^{\hdot}$]\label{Nonm}
By gluing disks to the outputs of the framed surface one obtains a morphism (here we also stabilise the resulting framed surface):
\begin{equation}\label{for}
\Pi_j\colon \overline{\mathcal N}^{fr}_{g,I,J}\longrightarrow \overline{\mathcal N}^{fr}_{g,I,J\setminus j}
\end{equation} 
Skew-symmetrising output boundaries one defines a hairy oriented Getzler-Kapranov complex $\textsf H_{\geq 1}\textsf {OGK}^{\hdot}_{1}$ (cf. \cite{AK}). Using compatibility between the operation $\Pi$ and an operation $\pi\colon \overline{\mathcal M}_{g,p}\longrightarrow \overline{\mathcal M}_{g,p-1}$ one shows the quasi-isomorphism between $\textsf H_{\geq 1}\textsf {OGK}^{\hdot}_{1}$ and a hairy Getzler-Kapranov complex $\textsf H_{\geq 1}\textsf {GK}^{\hdot}_0.$ 
Applying Theorem from \cite{AK} one computes (in genera $\geq 2$):
$$
\textsf H_{\geq 1}\textsf {OGK}^{\hdot}_{1}\cong \prod_{g\geq 2} C_c^{\hdot-1}(\mathcal M_g,\mathbb Q) 
$$
It is reasonable to call this complex non-labelled oriented Getzler-Kapranov complex and denote by $\textsf {OGK}_1^{\hdot}.$ Note that using Corollary 
\ref{we}  one reproves Theorem $3$ from \cite{AWZ} and gets the quasi-isomorphism between the oriented graph complex $\textsf {OGC}_1^{\hdot}$ and the Kontsevich graph complex $\textsf {GC}_0^{\hdot}$ \cite{Will2}.

\end{remark}

\section{The Merkulov-Willwacher conjecture} 

\subsection{Deformation complexes of properads} 

We will denote by $\textsf {Frob}$ we will denote the properad of Frobenius algebras i.e. algebras over this properad are commutative Frobenius algebras (all operations in this corresponding properad are given by $\mathbb Q$) with identical morphisms between them. We will also use a notation $\textsf {Frob}^{\diamond}$ for the properad of involutive Frobenius algebras (all operations are $\mathbb Q$, but all higher compositions are zero). These properads are equipped with a canonical morphism $\textsf {Frob}\longrightarrow \textsf {IFrob}.$ We will also use a notation $\textsf {AC}$ for the properad from \cite{AWZ}. Let $\textsf C$ be an augmented properad denote by $\Omega$ the properadic bar construction \cite{Vall}, which is defined as DG-properad structure on the free properad generated by the augmentation ideal of $\textsf C.$ We will use the following standard notations: $\Omega(\textsf {IFrob}^*):=\textsf {hoLieB}$ and $\Omega(\textsf {Frob}^*):=\textsf {hoLieB}^{\diamond}$
The properad $\textsf {hoLieB}$ (resp. $\textsf {hoLiB}^{\diamond}$) is a minimal resolution for the properad $\textsf {LieB}$ (resp.  $\textsf {LieB}^{\diamond}$) which controls Lie bialgebras (resp. involutive Lie bialgebras)\footnote{When we omit a subscript $d$ we tacitly assume that $d=0.$} \cite{CMW}. We have the corresponding Koszul dual morphism $\textsf {hoLieB}\longrightarrow \textsf {hoLieB}^{\diamond}$ to the morphism between Frobenius properads.

For a morphism of properad $f\colon \Omega(\textsf C)\longrightarrow \textsf P$ we will consider the corresponding deformation complex \cite{MV1}\cite{MV2} which is as mere graded vector spaces is:
$$
\mathrm {Def}(\Omega(\textsf C)\longrightarrow \textsf P):=\prod_{n,m} \mathrm {Hom}_{\Sigma_n\times \Sigma_m}(\textsf C(n,m),\textsf P(n,m))
$$
The right hand side is naturally a properad (a convolution properad) and hence by \cite{MK} carries a natural Lie bracket, such that $f$ is a Maurer-Cartan element. We assume that the corresponding differential in the deformation complex is twisted by $f.$ 

\subsection{Oriented Getzler-Kapranov complex as deformation complex} We will denote by $\overline{\mathcal N}^{fr}$ the DG-coproperad in with a space of $(n,m)$-operations given by the chains of moduli space of stable framed surfaces of an arbitrary genus i.e with $n+m$-framings:
$$
\overline{\textsf N}^{fr}(n,m)=\prod_{g\geq o} C^{\hdot}(\overline{\mathcal N}^{fr}_{g,[m],[n]},\mathbb Q)
$$
The composition law is given by gluing morphisms \eqref{glue}. Note that this DG-coproperad is equivalent to the DG-coproperad, underlying the DG-PROP associated with the Deligne-Mumford modular cooperad $C^{\hdot}(\overline{\mathcal M},\mathbb Q)$ \cite{HV}. Applying the formality result \cite{FORM} one shows that the coproperad $\overline{\mathcal N}^{fr}$ is formal. Note that the classical result about TQFTs states that there is a morphism of DG-coproperads:
\begin{equation}\label{tqfts}
\star\colon \textsf {Frob}^*\overset{\sim}{\longrightarrow} H^0(\overline{\mathcal N}^{fr},\mathbb Q)\hookrightarrow H^{\hdot}(\overline{\mathcal N}^{fr},\mathbb Q)\cong \overline{\textsf {N}}^{fr}
\end{equation} 
We will use the following notation
$$
\textsf B_g \textsf H\textsf H_S\textsf {OGK}^{\hdot}
$$
for the part of  $\mathrm {Def}(\Omega(\textsf {AC}^*)\overset{\star}{\longrightarrow} \Omega(\overline{\textsf N}^{fr}))$ which consists of decorated graphs of genus $g$ with $S$-marking. Following \cite{AWZ} the part of the deformation complex  $\mathrm {Def}(\Omega(\textsf {AC}^*)\overset{}{\longrightarrow} \textsf{hoLieB})$ which consists of graphs of genus $g$ with $S$-markings will be denoted by $\textsf B_g\textsf H\textsf H_S\textsf {OGC}^{\hdot}.$ We also employ analogous for the oriented Getzler-Kapranov complexes. Applying the functoriality of deformation complexes to \eqref{tqfts} for every $g\geq 0$ and a non-empty finite set $S$ such that $2g+|S|-2>0$ we get the following morphism:
$$
\textsf B_g\textsf H\textsf H_S\textsf {OGC}^{\hdot}\longrightarrow \textsf B_g\textsf H\textsf H_S\textsf {OGK}^{\hdot}
$$
We will show that this morphism induces the injection in cohomology. 
We have the following:
\begin{Th} There is a canonical morphism:
\begin{equation}\label{mor2}
 \textsf B_g\textsf H_S\textsf {OGK}^{\hdot}_1{\longrightarrow}\textsf B_g\textsf H\textsf H_S\textsf {OGK}^{\hdot},
 \end{equation} 
which induces injection in cohomology, such that the following diagram commutes:
\begin{equation}
\begin{diagram}[height=2.3em,width=2.3em]
 H^{\hdot}(\textsf B_g\textsf H_S\textsf {OGC}_{1}) & &  \rTo^{\sim} &  &  H^{\hdot}(\textsf B_g\textsf H\textsf H_S\textsf {OGC}) &  \\
\dTo_{}^{} & & &  & \dTo_ {}^{}  && \\
 H^{\hdot}(\textsf B_g\textsf H_S\textsf {OGK}_{1}) & &  \rInto^{}_{} &  & H^{\hdot}(\textsf B_g\textsf H\textsf H_S\textsf {OGK})
 \\
\end{diagram}
\end{equation}
\end{Th} 
\begin{proof} First we construct morphism \eqref{mor2} by the following rule:
\par\medskip 
For a directed graph $G$ we denote by the $G^{\neq in}$ the graph obtained by removing all incoming external edges and "stabilising" passing vertices. Consider the following morphism
$$\Pi\colon \prod_{v\in V(G)} \coprod_{g\geq 0} \overline{\mathcal N}_{g,in(v),out(v)}^{fr}\longrightarrow \prod_{v\in V(G^{\neq in})} \coprod_{g\geq 0} \overline{\mathcal N}_{g,in(v),out(v)}^{fr}$$
which is defined by gluing an analytic disk $\mathbf D$ at all inputs of the framed surface and stabilising it \eqref{for}. The pullback along $\Pi$ induces morphism \eqref{mor2}. We claim that this morphism induces the monomorphism in the cohomology. We have the following factorisation:
\begin{equation}
\begin{diagram}[height=2.9em,width=2.4em]
  \prod_{v\in V(G)} \coprod_{g\geq 0} \overline{\mathcal N}_{g,in(v),out(v)}^{fr} & &  \rTo^{\Pi} &  &  \prod_{v\in V(G^{\neq in})} \coprod_{g\geq 0} \overline{\mathcal N}_{g,in(v),out(v)}^{fr}&  \\
\dTo_{\sim}^{} & & &  & \dTo_ {\sim}^{}  && \\
 \prod_{v\in V(G)} \coprod_{g\geq 0} \overline{\mathcal M}_{g,in(v)\sqcup out(v)}^{fr} & &  \rTo^{\pi}_{} &  & \prod_{v\in V(G^{\neq in})} \coprod_{g\geq 0} \overline{\mathcal M}_{g,in(v)\sqcup out(v)}^{fr} \\
\end{diagram}
\end{equation}
Vertical arrows are defined by \eqref{fr}.\footnote{The spaces in the diagram are spaces of operations in the corresponding properad associated with a graph $G.$} Where: 
$$
\pi_{j}\colon \overline{\mathcal M}_{g,I\sqcup J}\longrightarrow \overline{\mathcal M}_{g,I\sqcup J\setminus j}
$$
is a morphism between Deligne-Mumford moduli stack which forgets markings and stabilise the resulting curve \cite{KN}. It is enough to show that the pullback of the morphism:
is injective in the cohomology. This follows from the fact that $\pi_{k+n!}(e(\pi_{k+n})\wedge \pi^*_{k+n})=\mathrm {id}\circ {const},$ where $e(\pi_{k+n})$ is the corresponding Euler class and $const$ is a non zero number. Hence we shown that \eqref{mor2} is injective in the cohomology.

\end{proof}

Hence we get the following:

\begin{Cor}\label{imp3} For every $g\geq 0$ and non empty finite set $S$ such that $2g+|S|-2>0$ we have the injective morphism:
$$
H^{\hdot}(\textsf B_g\textsf H\textsf H_S\textsf {OGC})\longrightarrow H^{\hdot}(\textsf B_g\textsf H\textsf H_S\textsf {OGK})
$$
\end{Cor} 
Under morphism \eqref{imp3} and Corollary \ref{we} this map correspond to the Chan-Galatius-Payne morphism \cite{CGP1}. 


%

\subsection{Properads or ribbon graphs } 

By a \textit{ribbon graph} $\Gamma$  we understand a triple $(H(\Gamma), \sigma_1, \sigma_0)$ where $H(\Gamma)$ is a finite set called the \textit{set of half
edges of} $\Gamma$, $\sigma_1\colon H(\Gamma)\longrightarrow  H(\Gamma)$ is a fixed point free involution and a permutation $\sigma_0 \colon H(\Gamma)\longrightarrow  H(\Gamma).$ Orbits of $\sigma_1$ are called the \textit{edges of} $\Gamma$, we will denote the set of edges of $\Gamma$ by $E(\Gamma).$  A set of orbits, $V(\Gamma):= H(\Gamma )/\sigma_0,$ is called the set of \textit{vertices of the ribbon graph} $\Gamma.$ We have a canonical map:
$$
p\colon H(\Gamma)\longrightarrow V(\Gamma):= H(\Gamma )/\sigma_0
$$
For each $v\in V(\Gamma)$ the pre-image of $p$ will be called a set of half-edges attached to a vertex $v:$
$$
p^{-1}(v):=h(v).
$$
The orbits of the permutation $\sigma_{2}:=\sigma_0^{-1}\sigma_1$ are called \textit{boundaries} of the ribbon graph $\Gamma.$ The set of boundaries of $\Gamma$ is denoted by $B(\Gamma)$. By the \textit{genus} of a ribbon graph, we understand the following quantity;
$$
g(\Gamma):=1+\frac{1}{2}(E(\Gamma)-V(\Gamma)-B(\Gamma))
$$
The definition implies that a ribbon graph $\Gamma$ is the same as a standard graph with a fixed cyclic structure on the set of half-edges $H_v(\Gamma)$ at each vertex $v\in V(\Gamma).$
\par\medskip

Recall some definitions from \cite{MW}. Let $d$ be an integer, denote by $\textsf{RGra}_d$ the \textit{properad of ribbon graphs}. This properad is defined by the collection of vector spaces $\{\textsf{RGra}_d(m,n)\}$ which consists of directed and connected ribbon graphs with $[n]$-labelled vertices and $[m]$-labelled boundaries with a certain choice of the orientation on the set of edges of a ribbon graph:
\begin{equation}\label{rb1}
\textsf{RGra}_d(m,n):=\bigoplus_{l\geq 0} \left(\mathbb Q\langle R_{m,n}^l \rangle \otimes_{P_l}
\mathrm{sgn}^{(d)}_l\right)[l(d-1)]
\end{equation}
where $R_{m,n}^l$ is a collection of ribbon graphs with $[n]$-labelled vertices and $[m]$-labelled boundaries and $[l]$-labelled edges. The permutation group $\Sigma_l$ acts on elements of $R_{m,n}^l$ by changing the orderings of edges, while the group $\Sigma_2^{\times l}$ acts by flipping the directions of edges. By $sgn_l^{(d)}$ we denote the one-dimensional representation of
the group $P_l:=\Sigma_l\times \Sigma_2^{\times l}$
on which $\Sigma_l$ acts trivially for $d$ odd and by sign for $d$ even, and each $\Sigma_2$ acts trivially for $d$ even and by sign for $d$ odd. Following \textit{ibid.} we define the composition:
\begin{equation}\label{rb2}
\circ\colon \textsf{RGra}_d(p,m)\otimes_{\mathbb Q} \textsf{RGra}_d(m,n) \longrightarrow \textsf{RGra}_d(p,n)
\end{equation}
As partial compositions $\circ_i$ which are defined by gluing the $i$-oriented boundary of $\Gamma\in \textsf{RGra}_d(m,n)$ to the $i$-vertex of a ribbon graph $\Gamma'\in \textsf{RGra}_d(p,m).$ Note that the operad $\textsf{RGra}_d$ is naturally graded by the genus of a ribbon graph. We have a natural genus grading on \eqref{rb1}. Moreover compositions $\circ_i$ respect this grading. Indeed $\Gamma'\circ_i\Gamma\in \textsf {RGra}_d(p+m-1,n+m-1)$ and since $|E(\Gamma)|+|E(\Gamma')|=|E(\Gamma'\circ_i\Gamma)|$ by formula \eqref{rb2} we get $g(\Gamma)+g(\Gamma')=g(\Gamma'\circ_i\Gamma).$ We call $\textsf{RGra}_d$ a \textit{properad of ribbon graphs}.This properad is equipped with a natural morphism:
\begin{equation}\label{cs1}
\diamond \colon \textsf {LieB}_{d,d}^{\diamond}\longrightarrow \textsf {RGra}_d.
\end{equation} 
This morphism is defined by by the rule:
$$
\diamond \colon [\,\,,\,\,]\longmapsto
\xy
 (0,0)*{\bullet}="a",
(5,0)*{\bullet}="b",
\ar @{-} "a";"b" <0pt>
\endxy \ \quad \diamond \colon \delta \longmapsto \ \xy
(0,-2)*{\bullet}="A";
(0,-2)*{\bullet}="B";
"A"; "B" **\crv{(6,6) & (-6,6)};
\endxy.
$$
Denote by $*\colon \textsf {LieB}_{d,d}^{\diamond}\longrightarrow \textsf {LieB}_{d,d}^{\diamond}\overset{\diamond}{\longrightarrow} \textsf {RGra}_d$ the composite morphism where the first arrow sendd cobracket to zero. We suggest the following:




\begin{Conj}\label{Conj1}
There is a morphism of properads (rather a correspondence):
\begin{equation}\label{mor1}
\Omega(\overline{\textsf N}^{fr})\longrightarrow \textsf {RGra}_0,
\end{equation} 
such that the following diagram commutes: 
\begin{equation}
\begin{diagram}[height=2.3em,width=2.3em]
\Omega(\textsf {Frob}^*) & &  \rTo^{\star} &  &   \Omega(\overline{\textsf N}^{fr}) &  \\
\dTo_{\sim}^{} & & &  & \dTo_ {\mathrm{}}^{}  && \\
\textsf {hoLieB}^{\diamond} & &  \rTo^{*}_{} &  & \textsf {RGra}_0   \\
\end{diagram}
\end{equation}
\end{Conj} 
Let us speculate a little about this conjecture:
\begin{remark}\label{mgr} For every $d\in \mathbb Z$ denote by $\textsf {GRav}_d$ the \textit{gravity properad} from \cite{Mer}. This DG-properad can be defined by applying T. Willwacher's properadic twisting construction \cite{Will}  to the properad $\textsf {RGra}_d.$ Hence it is equipped with a canonical morphism:
$$
\textsf{can}\colon \textsf {GRav}_d\longrightarrow \textsf{RGra}_d.
$$
The space of operations $\textsf {GRav}_d(m,n)$ of this properad is quasi-isomorphic by the chains of K. Costello's moduli spaces $D_{g,m,0,n}$ of nodal disks with $[n]$-marked points in the interior and $[m]$-labelled boundaries \cite{Cost}. 
I suspect that morphism \eqref{mor1} "factors" through the Gravity properad:
$$
\Omega(\overline{\textsf N}^{fr}) \longleftarrow \textsf {GRav}_0\longrightarrow \textsf{RGra}_0
$$
I expect that the left hand-side arrow is a (zigzag) quasi-isomorphism. Moreover there is a morphism of properads $\textsf {LieB}^{\diamond}_{0,0}\rightarrow \textsf {GRav}_0$ (the cobracket goes to zero) and one can prove that there is an injective morphism between the deformation complexes of $\textsf{RGra}_0$ and  $\textsf {GRav}_0.$ 

\end{remark} 

\begin{remark} Conjecture \ref{Conj1} can be viewed as a version of the correspondence between different TQFTs (cf. \cite{Cost1}). The properad $\textsf{RGra}_0,$ more precisely $\textsf {GRav}_0$ (see Remark \ref{mgr}), corresponds to moduli spaces of bordered surfaces (disks) with singularities on boundary. The properad  $\overline{\textsf N}^{fr}$ corresponds to bordered surfaces with singularities in interior. I expect that the correspondence \eqref{mor1} is induced by K. Costello's homotopy equivalence and factors through M. Liu's moduli space of bordered surfaces with all possible quadratic singularities allowed \cite{Liu}.

\end{remark} 

\begin{remark} One may state an analog of Conjecture \eqref{Conj1} by replacing a morphism $*$ with the Chas-Sullivan morphism $\diamond:$ 
\begin{equation}
\begin{diagram}[height=2.3em,width=2.3em]
\Omega(\textsf {Frob}^*) & &  \rTo^{\star} &  &   \Omega(\overline{\textsf N}^{fr}) &  \\
\dTo_{\sim}^{} & & &  & \dTo_ {\mathrm{}}^{}  && \\
\textsf {hoLieB}^{\diamond} & &  \rTo^{\diamond}_{} &  & \textsf {RGra}_0   \\
\end{diagram}
\end{equation}
There are two reasons (at least) to expect that this "stronger" statement is not true. First is that according to \cite{Mer} there is a morphism $\textsf {qLieB}\rightarrow H^{\hdot}(\textsf {GRav}),$ where $\textsf {qLieB}$ is the properad of quasi-Lie bialgebras in the sense of V. Drinfeld \cite{Drin4} and this morphism \textit{does not} factor through the properad of Lie bialgebras. The second reason is that the using Remark \ref{Nonm} from the proof of Theorem \ref{MW3} one surprisingly extracts the main result of \cite{AK4}.
\end{remark} 

To support this conjecture we do the following constructions. Recall that if we restrict the properad $\overline{\textsf N}^{fr}$ to the underlying cooperad we get the cooperad rom \cite{AV}. By Lemma from \textit{ibid} this cooperad is equivalent to hypercommutative cooperad $\textsf{Hycomm}$ i.e. a cooperad with $k$-operations given by $H^{\hdot}(\overline{\mathcal M}_{0,k},\mathbb Q).$ The latter cooperad is quasi-isomorphic to the bar construction applied to E. Getzler's operad gravity operad $\textsf{Gr}$ i.e. an operad with controls open locus of smooth rational curves \cite{GK}. The latter operad is equivalent to the operad of ribbon trees $\textsf {RTree}_0$ which is equivalent to the operadic part of $\textsf {RGra}_0$ i.e. we consider ribbon graphs of genus zero with exactly one boundary. Then the corresponding twisted operad is equivalent to the gravity operad $\textsf {Gr}$ \cite{Wa}. Moreover one has a morphism $\textsf {Comm}\longrightarrow \textsf{Hycomm}$ such that the Koszul dual morphism is $\textsf {Lie}\longrightarrow \textsf{Gr}$ Hence the get the proof of the following

\begin{Lemma} Conjecture \ref{Conj1} is true in genus zero i.e. the following square commutes:

\begin{equation}
\begin{diagram}[height=2.3em,width=2.3em]
\Omega(\textsf {Comm}) & &  \rTo^{} &  &  \Omega(\textsf {Hycomm})   &  \\
\dTo_{\sim}^{} & & &  & \dTo_ {}^{}  && \\
\textsf {Lie} & &  \rTo^{}_{} &  & \textsf {RTree}_0   \\
\end{diagram}
\end{equation}
\end{Lemma} 

Note that the deformation complex of propepad $\textsf {LieB}_{d,d}$ can be be identified with an oriented graph complex $\textsf {OGC}_{2d+1}^{\hdot}$ \cite{MW2}. Following \cite{MW} we denote the deformation complex of the morphism $\textsf {LieB}_{d,d}\rightarrow \textsf{RGra}_d$ by $\textsf {RGC}_d^{\hdot}(\delta+\Delta_1).$ Applying the functoriality of deformation complexes from \eqref{cs1} we get a morphism 
\begin{equation}\label{mw}
\textsf {OGC}_{2d+1}^{\hdot}\longrightarrow \textsf {RGC}_d^{\hdot}(\delta+\Delta_1)
\end{equation} 
Following \textit{ibid.} we have:
\begin{Conj}[S. Merkulov and T. Willwacher '15] Morphism \eqref{mw} induces a monomorphism in the cohomology.

\end{Conj} 

We have the following: 

\begin{Th}\label{MW3} Conjecture \ref{Conj1} together with Remark \ref{mgr} implie the Merkulov-Willwacher conjecture.

\end{Th} 

\begin{proof}  By Conjecture \ref{Conj1} and Remark \ref{mgr} we have the commutative diagram:

\begin{equation*}
\begin{diagram}[height=2.3em,width=2.3em]
 \mathrm {Def}(\Omega(\textsf {AC})^*\rightarrow \textsf {LieB}) & &  \rTo^{\sim} &  & \mathrm {Def}(\Omega(\textsf {AC})^*\rightarrow \textsf {LieB}) & &  \lTo^{\sim} &  &   \mathrm {Def}(\Omega(\textsf {AC})^*\rightarrow \textsf {LieB}) &  \\
\dTo_{}^{} & & &  &\dTo_{}^{} & & &  & \dTo_ {}^{}  && \\
 \mathrm {Def}(\Omega(\textsf {AC})^*\rightarrow \overline{\textsf {N}}^{fr}) & &  \rTo_{\sim} &  &  \mathrm {Def}(\Omega(\textsf {AC})^*\rightarrow \textsf {GRav}) & &  \lInto &  &  \mathrm {Def}(\Omega(\textsf {AC})^*\rightarrow \textsf {RGra})   \\
\end{diagram}
\end{equation*}
By Corollary \ref{imp3} the left vertical arrows is injective and moreover coincides with the Chan-Galatius-Payne morphism. Hence we get a proof of Conjecture $25$ from \cite{AWZ}. Further by \cite{AK4} and \cite{AK} we get the result.

\end{proof}

\bibliographystyle{alphanum}
\bibliography{tt}

\end{document}